\documentclass[a4paper]{amsart}

\usepackage[all,cmtip]{xy}

\renewcommand{\H}{\mathbb{H}}
\newcommand{\R}{\mathbb{R}}
\newcommand{\Rep}{\mathcal{R}}
\renewcommand{\S}{\Sigma}
\newcommand{\G}{G}
\newcommand{\K}{K}
\newcommand{\SO}{\mathrm{SO}}
\newcommand{\X}{\mathfrak{X}}
\newcommand{\q}{\textbf{q}}
\renewcommand{\P}{\mathbb{P}}
\newcommand{\T}{\mathcal{T}}

\newcommand{\sslash}{\mathbin{/\mkern-6mu/}}

\theoremstyle{definition}
\newtheorem{defi}{Definition}[section]
\theoremstyle{plain}
\newtheorem{theo}{Theorem}[section]
\newtheorem*{Theo}{Main Theorem}

\newtheorem{prop}[defi]{Proposition}
\newtheorem{cor}[defi]{Corollary}

\newtheorem{conj}[defi]{Conjecture}
\theoremstyle{remark}
\newtheorem{rem}[defi]{Remark}
\newtheorem{ex}[defi]{Example}

\title[Maximal representations in rank 2]{Higgs bundles, pseudo-hyperbolic geometry and maximal representations}
\author{J\'{e}r\'{e}my Toulisse}
\address{
Department of Mathematics \\
University of Southern California \\
3620 S. Vermont Avenue, KAP 104 \\
Los Angeles, CA 90089-2532}
\email{toulisse@usc.edu}
\date{\today}

\begin{document}
\maketitle

\begin{abstract}
These notes are an extended version of a talk given by the author in the seminar \textit{Th\'eorie Spectrale et G\'eom\'etrie} at the \textit{Institut Fourier} in November 2016. We present here some aspects of a work in collaboration with B. Collier and N. Tholozan \cite{CTT}. We describe how Higgs bundle theory and pseudo-hyperbolic geometry interfere in the study of maximal representations into Hermitian Lie groups of rank $2$.
\end{abstract}

\tableofcontents

\section*{Introduction}

In this paper, $\S$ will denote a closed oriented surface of genus $g>1$ and $\Gamma$ its fundamental group. 

The Teichm\"uller space $\T(\S)$ is the space of isotopy classes of complex structures on $\S$. By the Uniformization Theorem, $\T(\S)$ is identified with the space of hyperbolic structures on $\S$ (modulo isotopy). Each hyperbolic surface being isometric to the quotient of the hyperbolic disk by a discrete subgroup of the isometry group $\text{Isom}^+(\H^2)$, $\T(\S)$ is identified with the space of conjugacy classes of discrete and faithful representations $\rho: \Gamma \to \text{Isom}^+(\H^2)$ (also called \textit{Fuchsian representations}).

By a result of Teichm\"uller, $\T(\S)$ is diffeomorphic to $\R^{6g-6}$. Because $\T(\S)$ is the deformation space of complex structures on $\S$, it is naturally a complex manifold. The corresponding complex structure is compatible with the Goldman symplectic form, turning $\T(\S)$ into a K\"ahler manifold (the corresponding K\"ahler metric is the Weil-Petersson metric).

The concept of maximal representation of $\Gamma$ into a Hermitian Lie group $G$ provides a  natural generalization of Fuchsian representations in higher rank Lie groups. These representations share many nice geometrical and dynamical properties with Fuchsian representations. However, the intricate geometry of the corresponding symmetric spaces make the study of these representations difficult. Note also that the set $\Rep^{max}(\Gamma,G)$ of conjugacy classes of maximal representations is generally not connected and carries no ``natural'' complex structure.

In these notes, we are interested in maximal representations $\rho : \Gamma \to G$, where $G$ is a Hermitian Lie group of rank $2$. By rigidity results of S. Bradlow, O. Garc\'ia-Prada and P. Gothen \cite{MaxRepsHermSymmSpace} (see also M. Burger, A. Iozzi and A. Wienhard \cite{burgeriozziwienhard} for the general case), it is sufficient to understand maximal representations into $G:=\SO_0(2,n+1)$ for $n>1$.

Our main idea is to consider the action of maximal representations on the pseudo-hyperbolic space $\H^{2,n}$. This is a pseudo-Riemannian symmetric space of signature $(2,n)$ and constant curvature $-1$. Because of this curvature property, its geometry is somehow easier to understand than that of the Riemannian symmetric space.

Recall that an embedding into a pseudo-Riemannian manifold is \textit{space-like} if the induced metric is Riemannian and is \textit{maximal} if any local deformation decreases the area.

The main result is the following:

\begin{Theo}
For any maximal representation $\rho: \Gamma \to G$, there exists a unique $\rho$-equivariant maximal space-like embedding $u: \widetilde\S \to \H^{2,n}$.
\end{Theo}

The existence part uses the theory of Higgs bundles. More precisely, we show that the Higgs bundle associated to a $\rho$-equivariant minimal surface into the Riemannian symmetric space is \textit{cyclic}. This property gives a reduction of the associated flat $\SO_0(2,n+1)$ bundle to a $\SO(2)\times \text{S}(\text{O}(1)\times \text{O}(n))$-bundle. From this reduction, one can extract a $\rho$-equivariant map $u: \widetilde\S\to \H^{2,n}$ that we show is a maximal space-like embedding.

The uniqueness is proved using pseudo-hyperbolic geometry. Supposing the existence of two different $\rho$-equivariant maximal space-like embeddings, one can define a \textit{causal distance} as the distance given by time-like geodesics. By equivariance, one can find a time-like geodesic segment between the two surfaces whose causal length is maximal. One can thus use the negative curvature of $\H^{2,n}$ to get a contradiction. This contradiction is obtained using a maximum principle, and generalizes a result of F. Bonsante and J.-M. Schlenker \cite{BonsanteSchlenker}.

Finally, the Gauss map of the $\rho$-equivariant maximal surface $u: \widetilde\S \to \H^{2,n}$ gives a $\rho$-equivariant minimal surface in the Riemannian symmetric space.

This result has important consequences:

\begin{enumerate}
	\item It gives an explicit interpretation of all the connected components of $\Rep^{max}(\Gamma,G)$.
	\item It proves (a generalization of) a conjecture of Labourie for maximal representations in rank $2$.
	\item It provides a natural complex structure on $\Rep^{max}(\Gamma,G)$.
	\item It gives a way to construct geometric structures associated to maximal representations (this aspect will not be treated here).
\end{enumerate}

\begin{rem}\label{r:AdScase} The case $G=\SO_0(2,2)$ is very particular. In fact, up to finite index, $\SO_0(2,2)\cong \SO_0(2,1)\times \SO_0(2,1)$ and we get an identification $\Rep^{max}(\Gamma,G)=\T(\S)\times \T(\S)$. From the work of G. Mess \cite{Mess}, maximal representations in $\SO_0(2,2)$ are holonomies of Globally Hyperbolic $\H^{2,1}$-structures on $\S\times \R$ (also called \textit{anti-de Sitter}). In this setting, our Main Theorem was proved by T. Barbot, F. B\'eguin et A. Zeghib \cite{BBZ}. The corresponding equivariant minimal surface in the Riemannian symmetric space $\H^2\times\H^2$ is the graph of a minimal Lagrangian diffeomorphism  between the corresponding hyperbolic surfaces (whose existence is due to R. Schoen \cite{schoen}). The link between the two objects was made by K. Krasnov and J.-M. Schlenker in \cite{krasnovschlenker}.
\end{rem}

\noindent\textbf{Organization of the paper:} The first section describes maximal representations into a Hermitian Lie group. The second section introduces Higgs bundles for $G=\SO_0(2,n+1)$ and explain the non-Abelian Hodge correspondence. The third section is devoted to the proof of the Main Theorem, while the fourth one describes all the consequences.

\section{Maximal representations}

\subsection{A main example: the Teichm\"uller space}

Denote by $\R^{2,1}$ the usual $3$-dimensional space $\R^3$ endowed with the signature $(2,1)$ quadratic form $\q$ defined by $\q(x)=x_1^2+x_2^2-x_3^2$.

 The Lie group $\SO_0(2,1)$, which is the connected component of the identity of the group of orthogonal transformations of $\R^{2,1}$, acts by isometries on the hyperbolic disk $\H^2$. Recall that
\[ \H^2= \P \big(\{x\in \R^{2,1},~\q(x)<0\} \big)\subset \R\P^2, \]
where $\P: \R^3\setminus\{0\} \to \R\P^2$ is the natural projection.

The corresponding character variety 
\[\Rep(\Gamma,\SO_0(2,1)):=\text{Hom}(\Gamma,\SO_0(2,1))\sslash\SO_0(2,1)\]
is the quotient (in the algebraic sense) of the set of morphisms from $\Gamma$ to $\SO_0(2,1)$ by the action of $\SO_0(2,1)$ by conjugation. It is well-known that any conjugacy class $[\rho]\in \Rep(\Gamma,\SO_0(2,1))$ contains a reductive representative.

Given a morphism $\rho: \Gamma \to \SO_0(2,1)$, there always exists a smooth map $f: \widetilde{\S} \to \H^2$ which is $\rho$-equivariant, that is satisfies:
\[ \forall x\in \widetilde \S, \forall \gamma\in \Gamma,~f(\gamma.x)=\rho(\gamma)f(x). \]
Here, $\Gamma$ acts on $\widetilde{\Sigma}$ by deck transformations.

Given such a $\rho$-equivariant map $f: \widetilde\S \to \H^2$, the pull-back $f^*\textbf{vol}$ of the volume form on $\H^2$ by $f$ is invariant under the action of $\Gamma$. In particular, $f^*\textbf{vol}$ descends to a closed $2$-form on $\Sigma$, and one can define the \textit{Euler class} of $\rho$ by
\[ e(\rho):= \frac{1}{2\pi}\int_\S f^*\textbf{vol}.\]

The Euler class is an integer, independent of the $\rho$-equivariant map $f$ and on the conjugacy class of $\rho$. In particular, the Euler class defines a continuous map
\[ e:~ \Rep(\Gamma,\SO_0(2,1)) \longrightarrow \mathbb{Z}.\]
The Euler class satisfies the Milnor-Wood inequality \cite{milnor}:
\[\forall[\rho]\in\Rep(\Gamma,\SO_0(2,1)),~\vert e([\rho])\vert \leq 2g-2.\]

In his thesis, W. Goldman proved the following:
\begin{theo}[W. Goldman]
The set $e^{-1}(2g-2)\subset \Rep(\Gamma,\SO_0(2,1)$ is canonically identified with the Teichm\"uller space $\T(\S)$ of $\S$.
\end{theo}

In particular, representations $\rho: \Gamma\to\SO_0(2,1)$ maximizing the Euler class are holonomies of hyperbolic metrics on $\S$ (the so-called \textit{Fuchsian representations}). Note also that, by the uniformization theorem, the Teichm\"uller space $\T(\S)$ is also the deformation space of complex structures on $\S$ and as such, is naturally a complex manifold (which is not granted from the representation point of view).

\subsection{Generalizing the Teichm\"uller space}

A natural way to generalize the Teichm\"uller space of $\S$ is to replace the group $\SO_0(2,1)$ by another semi-simple Lie group $G$. In that case, the character variety $\Rep(\Gamma,G)$ is defined in an analogous way, and one wants to decide which representations will play the role of the Fuchsian ones.

If $G$ is of non-compact type and $\K \subset G$ is a maximal compact subgroup, we denote by $\X:= G/\K$ the associated symmetric space. Recall that the Killing form of the Lie algebra $\frak{g}$ of $G$ naturally defines a Riemannian metric on $\X$. With this metric, $\X$ is non-positively curved and the \textit{rank of $G$}, denoted $\text{rk}(G)$, is defined to be the largest $k\in \mathbb{N}$ such that the Euclidean space $\R^k$ embeds isometrically in $\X$.

\begin{defi}
The Lie group $\G$ is of \textit{Hermitian type} if the associated symmetric space $\X_\G$ is a K\"ahler manifold. If $\G$ is of Hermitian type, we denote by $\omega$ the corresponding symplectic form on $\X$.
\end{defi}

\begin{ex}
The groups $\textrm{Sp}(2n,\R),~\textrm{SU}(p,q)$ (with $p,q>0$) and $\SO_0(2,n)$ are of Hermitian type.
\end{ex}

Given a Lie group of Hermitian type $G$ and a morphism $\rho: \Gamma \to \G$, one can mimic the definition of the Euler class and define the \textit{Toledo invariant} of $\rho$ by
\[ \tau(\rho):= \frac{1}{2\pi}\int_\S f^*\omega.\]
Here, $f: \widetilde\S \to \X$ is a $\rho$-equivariant smooth map.

We have the following:

\begin{prop}[\cite{BIWmaximalrep}]
With the proper normalization of the symplectic form $\omega$, the Toledo invariant is an integer valued function and satisfies the \textit{Milnor-Wood inequality}
\[\forall [\rho]\in\Rep(\Gamma,\G),~\vert\tau([\rho])\vert\leq (2g-2)\text{rk}(G).\]
\end{prop}

Goldman's Theorem motivates the following definition:

\begin{defi}
A representation $\rho: \Gamma \to G$ is \textit{maximal} if $\tau(\rho)=\text{rk}(G)(2g-2)$. We denote by $\Rep^{max}(\Gamma,\G)$ the set of conjugacy classes of maximal representations in $\G$.
\end{defi}

The set $\Rep^{max}(\Gamma,G)$ is a union of connected components of $\Rep(\Gamma,G)$ and gives a natural generalization of the classical Teichm\"uller space.

Maximal representations share many geometric and dynamical properties with Fuchsian representations. More precisely, any maximal representation $\rho: \Gamma \to G$ is \textit{Anosov} \cite{BILW} (as introduced by F. Labourie \cite{labouriehyperconvex} and O. Guichard, A. Wienhard \cite{wienhardanosov}) and so acts properly discontinuously and co-compactly on some open set $\Omega_\rho\subset G/P$, where $G/P$ is a flag manifold (that is, $P\subset G$ is a parabolic subgroup). In particular, maximal representations are holonomies of geometric structures.

In \cite{CTT}, we use the techniques presented here to explicitly describe the quotient $\Omega_\rho/\rho(\Gamma)$.

In these notes, we are interested in maximal representations into Hermitian Lie groups of rank $2$. Up to covering, those groups are $\textrm{PSp}(4,\R),~\textrm{PU}(2,n)$ and $\SO_0(2,n)$ (with $n\geq 2)$. The following is a rigidity result for maximal representations

\begin{theo}[\cite{burgeriozziwienhard}] Given a maximal representation $\rho: \Gamma \to G$ into a Hermitian Lie group $G$, there exists a subgroup of tube type $H\subset G$ and a compact subgroup $K\subset G$ such that $\rho(\Gamma)\subset H\times K \subset G$.
\end{theo}

Tube type Hermitian Lie groups of rank $2$ are (up to covering) $\textrm{PSp}(4,\R),~\textrm{PU}(2,2)$ and $\SO_0(2,n)$. Using the exceptional isomorphisms $\textrm{PSp}(4,\R)\cong \SO_0(2,3)$ and $\textrm{PU}(2,2)\cong \SO_0(2,4)$ (up to index 2), we obtain the following:

\begin{cor}\label{cor:maximalrepinrank2}
Up to a compact factor, any maximal representation $\rho: \Gamma \to G$ into a rank 2 Hermitian Lie group $G$ factors through $\SO_0(2,n)$ for some $n\geq 2$.
\end{cor}

In particular, we will study here maximal representations into $G=\SO_0(2,n+1)$. As explained in Remark \ref{r:AdScase}, the case $n=1$ is very special, so we will assume $n\geq 2$.

\section{Higgs bundles}

Here we describe the non-Abelian Hodge correspondence for the Lie group $G:=\SO_0(2,n+1)$. Recall that $G$ is the connected component of the identity of the group of orthogonal transformations of $\R^{2,n+1}$. The subgroup $K:=\SO(2)\times \SO(n+1)\subset G$ is a maximal compact subgroup and we denote by $\frak{X}:=G/K$ the associated symmetric space. A point $x\in \frak{X}$ corresponds to an orthogonal splitting $\R^{2,n+1}=\R^{2,0}\oplus\R^{0,n+1}$.

Given a Riemann surface structure $X\in\T(\S)$ on $\S$, we denote by $\mathcal{O}$ the trivial holomorphic vector bundle over $X$ and by $\mathcal{K}$ the canonical bundle. The slope of a holomorphic vector bundle $\mathcal{E}\to X$ is $\mu(\mathcal{E})=\frac{\text{deg}(\mathcal{E})}{\text{rk}(\mathcal{E})}$ where $\text{deg}(\mathcal{E})$ and $\text{rk}(\mathcal{E})$ are respectively the degree and the rank of $\mathcal{E}$.

\begin{defi}
A \textit{$\textrm{GL}(n,\mathbb{C})$-Higgs bundle on $X$} is a pair $(\mathcal{E},\Phi)$, where $\mathcal{E}$ is a rank $n$ holomorphic vector bundle over $X$, and $\Phi\in H^0(X,\mathcal{K}\otimes \text{End}(\mathcal{E}))$.

Such a Higgs bundle is \textit{stable} if any $\Phi$-invariant holomorphic sub-bundle $\mathcal{F}\subset \mathcal{E}$ satisfies $\mu(\mathcal{F})<\mu(\mathcal{E})$. It is \textit{poly-stable} if it is the direct sum of stable $\textrm{GL}(n_j,\mathbb{C})$ Higgs bundles with same slope.
\end{defi}

We are interested here in $G$-Higgs bundle:

\begin{defi}
A \textit{$G$-Higgs bundle over $X$} is a $\textrm{GL}(n+3,\mathbb{C})$-Higgs bundle $(\mathcal{E},\Phi)$ such that:
\begin{itemize}
	\item $\mathcal{E}$ reduces to a holomorphic $\SO(2,\mathbb{C})\times\SO(n+1,\mathbb{C})$ bundle. That is, $\mathcal{E}=\mathcal{U} \oplus \mathcal{W}$, where $\mathcal{U}$ is a rank $2$ holomorphic vector bundle with trivialization $\det\mathcal{U}=\mathcal{O}$ and a non-generate holomorphic quadratic form $q_{\mathcal{U}}$ (and similarly for $\mathcal{W}$).
	\item $\Phi$ writes $\left( \begin{array}{ll} 0 & \eta^\dagger \\ \eta & 0\end{array}\right)$ in the splitting $\mathcal{E}=\mathcal{U}\oplus \mathcal{W}$, where $\eta\in H^0(X,\mathcal{K}\otimes \text{Hom}(\mathcal{U},\mathcal{W}))$ and $\eta^\dagger$ is defined by the equation $q_\mathcal{W}(\eta x,y)=q_{\mathcal{U}}(x,\eta^\dagger y)$.
\end{itemize}
A $G$-Higgs bundle is stable (respectively poly-stable), if it is as a $\textrm{GL}(n+3,\mathbb{C})$-Higgs bundle.
\end{defi}

Schematically, a $G$ Higgs bundle $(\mathcal{E},\Phi)$ will be written as follow:
$$(\mathcal E,\Phi) = \xymatrix{\mathcal{U} \ar@/_/[r]_\eta & \mathcal{W} \ar@/_/[l]_{\eta^\dagger}}.$$

Using gauge theory, one can construct the moduli space $\mathcal{M}(X,G)$ of poly-stable $G$-Higgs bundles over $X$. The main interest for us is the following correspondence, called \textit{non-Abelian Hodge correspondence}:

\begin{theo}[Non-Abelian Hodge correspondence]
For any $X\in \T(\S)$, there is a canonical bijection $\Rep(\Gamma,G)\cong \mathcal{M}(X,G)$.
\end{theo}

\begin{proof}[Sketch of proof]
Fix $\rho\in\Rep(\Gamma,G)$ and consider the associated vector bundle 
\[E_\rho:= \big( \widetilde\S\times \R^{2,n+1}\big)/\Gamma,\]
where $\Gamma$ acts diagonally on $\widetilde\S\times \R^{2,n+1}$ (by deck transformations on $\widetilde\S$ and via $\rho$ on $\R^{2,n+1}$). The bundle $E_\rho$ is naturally equipped with a flat connection $\nabla$ and a $\nabla$-parallel signature $(2,n+1)$ metric $g$. The holonomy of $\nabla$ corresponds to $\rho$.

A reduction of $E_\rho$ to a $\SO(2)\times\SO(n+1)$ bundle is given by a splitting $E_\rho=U\oplus V$, where $U$ is a positive definite rank $2$ vector bundle and $V=U^\bot$. Such a splitting defines a metric $g_U\oplus (-g_V)$ on $E_\rho$ and is given by a $\rho$-equivariant map $f: \widetilde\S \to \frak{X}$ where $\frak{X}$ is the symmetric space of $G$.

Given a Riemann surface structure $X\in\T(\S)$, Corlette's Theorem \cite{corlette} implies the existence of a unique $\rho$-equivariant map $f:\widetilde X\to \frak{X}$ which is harmonic, that is, a map such that the $(1,0)$-part of the differential $df$ is a holomorphic section of $T^*X\otimes f^*T\frak{X}\otimes \mathbb{C}$.

In particular, for each $X\in\T(\S)$, we get a canonical splitting $E_\rho=U\oplus V$ as above. In such a splitting, the connection $\nabla$ decomposes as
\[\nabla = A + \Psi,\]
where $A=\left(\begin{array}{ll}A^U & 0 \\ 0 & A^V \end{array}\right)$ is unitary with respect to the metric $g_U\oplus(-g_V)$, and $\Psi= \left(\begin{array}{ll} 0 & \psi^\dagger \\
\psi & 0 \end{array}\right)$, where $\psi^\dagger$ satisfies $-g_V(\psi x,y)=g_U(x,\psi^\dagger y)$.

The complexification $\mathcal{E}:=E_\rho\otimes \mathbb{C}$ together with the holomorphic structure defined by the $(0,1)$ part of $A$ and the $\mathbb{C}$-linear extension $q:=q_\mathcal{U}\oplus q_\mathcal{W}$ of $g_U\oplus(-g_V)$ is a $\SO(2,\mathbb{C}\times \SO(n+1,\mathbb{C})$ holomorphic vector bundle. The $(1,0)$ part $\Phi$ of $\Psi$ is holomorphic, so the pair $(\mathcal{E},\Phi)$ is a $G$-Higgs bundle.

Denoting by $\lambda: \mathcal{E}\to\mathcal{E}$ the anti-linear involution given by conjugation, we obtain a Hermitian metric $h$ on $\mathcal{E}$ defined by $h(x,y)=q(x,\lambda y)$. The flatness of $\nabla$ gives the \textit{Hitchin equations}
$$F_A+[\Phi,\Phi^{*h}]=0,$$
 where $F_A$ is the curvature of the unitary connection $A$ and $\Phi^{*h}$ is defined by $h(\Phi x,y)=h(x,\Phi^{*h}y)$. Hitchin-Simpson Theorem \cite{simpsonVHS,hitchinselfduality} thus implies that $(\mathcal{E},\Phi)$ is poly-stable.

\medskip

Conversely, given a poly-stable $G$-Higgs bundle $(\mathcal{E},\Phi)$, there exists a unique Hermitian metric $h$ on $\mathcal{E}$ such that $F_A+[\Phi,\Phi^{*h}]=0$. Writing $h(x,y)=q(x,\lambda y)$ for some anti-linear involution $\lambda: \mathcal{E} \to \mathcal{E}$ preserving the splitting $\mathcal{E}=\mathcal{U}\oplus\mathcal{W}$, we get that the flat connection $\nabla:= A+\Phi+\Phi^{*h}$ preserves the real sub-bundle $E_\lambda:=\text{Fix}(\lambda)\subset \mathcal{E}$. This sub-bundle is endowed with a signature $(2,n+1)$ metric $g:=q_\mathcal{U}\oplus (-q_\mathcal{W})$ and the holonomy of $\nabla$ gives a reductive representation $\rho: \Gamma \to G$.
\end{proof}

One can actually go further: given a poly-stable $G$-Higgs bundle $(\mathcal{E},\Phi)$ over $X$, the $\SO(2,\mathbb{C})$ bundle $(\mathcal{U},q_\mathcal{U})$ can always be decomposed as $\left(\mathcal{L}\oplus\mathcal{L}^{-1},\left(\begin{array}{ll} 0 & 1 \\ 1 & 0 \end{array}\right) \right)$ where $\mathcal{L}$ is a holomorphic line bundle. Moreover, the degree of $\mathcal{L}$ is equal to half of the Toledo invariant of the associated representation $\rho\in\Rep(\Gamma,G)$.

In particular, one can decomposes $\eta=(\alpha,\gamma): \mathcal{L}\oplus\mathcal{L}^{-1}\to \mathcal{W}\otimes \mathcal{K}$ and $\eta^{\dagger}=(\gamma^\dagger,\alpha^\dagger): \mathcal{W}\to \big(\mathcal{L}\oplus\mathcal{L}^{-1}\big)\otimes \mathcal{K}$.

When the Higgs bundle $(\mathcal{E},\Phi)$ is associated to a maximal representation, the degree of $\mathcal{L}$ is $2g-2$. By stability, both $\alpha$ and the composition 
$$\mu:~\mathcal{L} \overset{\alpha}{\longrightarrow} \mathcal{W}\otimes \mathcal{K} \overset{\alpha^\dagger\otimes Id}{\longrightarrow} \mathcal{L}^{-1}\mathcal{K}^2$$
 are non-vanishing. In particular, $\mu\in H^0(X,\mathcal{K}^2\mathcal{L}^{-2})$ is a non-vanishing section of a degree $0$ line bundle, so $\mathcal{K}^2\mathcal{L}^{-2}\cong\mathcal{O}$. Moreover, $\mathcal{I}:= \text{Im}(\alpha)\otimes \mathcal{K}^{-1}\subset \mathcal{W}$ is a non-isotropic line sub-bundle with $\mathcal{I}^2=\mathcal{O}$. Note also that $\mathcal{L}\cong \mathcal{I}\mathcal{K}$.

Setting $\mathcal{V}$ the orthogonal of $\mathcal{I}$ in $\mathcal{W}$ and writing $\gamma = (q_2,\beta):~\mathcal{I}\mathcal{K}^{-1} \to \big(\mathcal{I}\oplus \mathcal{V} \big)\otimes \mathcal{K}$, one obtains the following corollary:

\begin{cor}\label{cor:maximalhiggsbundles}
A $G$ Higgs bundle $(\mathcal{E},\Phi)$ over $X$ associated to a maximal representation $\rho: \Gamma \to G$ decomposes as follow:

\[(\mathcal{E},\Phi)= \xymatrix{
\mathcal{I}\mathcal{K} \ar@/_/[r]_1 & \mathcal{I}  \ar@/_/[r]_1 \ar@/_/[l]_{q_2} & \mathcal{I}\mathcal{K}^{-1} \ar@/_/[l]_{q_2} \ar@/^/[ld]^{\beta} \\
& \mathcal{V} \ar@/^/[ul]^{\beta^\dagger} }.\]
Here $q_2=\frac{1}{2}\text{tr}(\Phi^2)\in H^0(X,\mathcal{K}^2)$ is a quadratic differential, $\mathcal{I}$ is a square root of the trivial bundle and $\beta\in H^0(X,\mathcal{I}\mathcal{K}^2\mathcal{V})$.
\end{cor}

The quadratic differential $q_2\in H^0(X,\mathcal{K}^2)$ in the previous parametrization has a very nice geometric interpretation. Up to a factor $\frac{1}{2}$, it is the \textit{Hopf differential} of the associated equivariant harmonic map $f: \widetilde X \to \frak{X}$ given by Corlette's Theorem. It measures the lack of conformality of the harmonic map. By a theorem of Eells-Sampson \cite{eellssampson}, the Hopf differential vanishes if and only if $f: \widetilde X \to \frak{X}$ is a minimal immersion. 

\section{Proof of the main Theorem}

\subsection{Pseudo-hyperbolic geometry}

Denote by $\R^{2,n+1}$ the real vector space $\R^{n+3}$ endowed with the quadratic form $\q$ of signature $(2,n+1)$ given by $\q(x)=x_1^2+x_2^2-x_3^2-...-x_{n+3}^2$. The group $G=\SO_0(2,n+1)$ is thus the connected component of the identity of the group of orthogonal transformations of $\R^{2,n+1}$.

The \textit{pseudo-hyperbolic space $\H^{2,n}$} is the set of negative-definite lines in $\R^{2,n+1}$:
\[\H^{2,n}:= \P \big(\{x\in\R^{2,n+1},~\q(x)<0\} \big),\]
where $\P:~\R^{n+3}\setminus\{0\} \to \R\P^{n+2}$ is the usual projection.

The quadratic form $\q$ induces a signature $(2,n)$ metric of constant curvature $-1$ on $\H^{2,n}$. The group $G$ acts by isometries on $\H^{2,n}$.

The boundary $\partial\H^{2,n}$ of $\H^{2,n}$ in $\R\P^{n+2}$ is usually called the \textit{Einstein Universe}. It corresponds to the set of isotropic lines in $\R^{2,n+1}$ and carries a natural conformally flat metric of signature $(1,n)$.

The projection $\P:~\R^{n+3}\setminus\{0\} \to \R\P^{n+2}$ restricts to a $2$-to-$1$ covering from the quadric $\widehat\H^{2,n}:=\{x\in\R^{2,n+1},~\q(x)=-1\}$ onto $\H^{2,n}$. 

Complete geodesics in $\H^{2,n}$ are the intersection of projective lines in $\R\P^{n+2}$ with $\H^{2,n}$. The quadratic form $\q$ on $\R^{2,n+1}$ gives a trichotomy for geodesics in $\H^{2,n}$:
\begin{itemize}
	\item \textit{Space-like geodesics} are the projection of signature $(1,1)$-planes in $\R^{2,n+1}$. They intersect $\partial\H^{2,n}$ in two distinct points and the restriction of the metric is positive definite.
	\item \textit{Time-like geodesics} are projection of signature $(0,2)$-planes in $\R^{2,n+1}$. They do not intersect the boundary $\partial\H^{2,n}$ and the restriction of the metric is negative definite.
	\item \textit{Light-like geodesics} are projection of degenerate planes in $\R^{2,n+1}$. They intersect $\partial\H^{2,n}$ in one point and the restriction of the metric is degenerate.
\end{itemize}

In particular, two points $x,y\in\widehat\H^{2,n}$ are on a space-like (respectively time-like and light-like) geodesic if an only if $\vert \langle x,y\rangle\vert >1$ (respectively $\vert \langle x,y\rangle\vert <1$ and $\vert \langle x,y\rangle\vert =1$), where $\langle .,.\rangle$ is the restriction of the quadratic form on $\widehat\H^{2,n}\subset \R^{2,n+1}$.

\medskip

\noindent\textbf{Maximal immersions.} An immersion $u: S \to \H^{2,n}$ from a surface $S$ into $\H^{2,n}$ is called \textit{space-like} if the induced metric is positive definite. A \textit{maximal space-like immersion} is a space-like immersion $u: S\to \H^{2,n}$ such that any local deformation strictly decreases the area of the induced metric. As in Riemannian geometry, an immersion is maximal if and only if the mean curvature vector field vanishes everywhere.

The following can be found in \cite{eellssampson} in the Riemannian setting:

\begin{prop}[Ells-Sampson]
Given a Riemann surface structure $X\in\T(\S)$ on $\S$, a space-like immersion $u: X \to \H^{2,n}$ is maximal if and only if it is harmonic and conformal.
\end{prop}

\subsection{Existence part}\label{s:existence}

Here, we prove the existence part of the Main Theorem:

\begin{prop}
Given a maximal representation $\rho\in\Rep^{max}(\Gamma,G)$, there exists a $\rho$-equivariant maximal space-like immersion $u:\widetilde\S\to\H^{2,n}$.
\end{prop}

\begin{proof}
Recall from Corollary \ref{cor:maximalhiggsbundles} that, given a Riemann surface structure $X\in\T(\S)$ on $\S$, the $G$ Higgs bundle $(\mathcal{E},\Phi)$ associated to $\rho$ has the form
\[(\mathcal{E},\Phi)= \xymatrix{
\mathcal{I}\mathcal{K} \ar@/_/[r]_1 & \mathcal{I}  \ar@/_/[r]_1 \ar@/_/[l]_{q_2} & \mathcal{I}\mathcal{K}^{-1} \ar@/_/[l]_{q_2} \ar@/^/[ld]^{\beta} \\
& \mathcal{V} \ar@/^/[ul]^{\beta^\dagger} }.\]
The quadratic differential $q_2\in H^0(X,\mathcal{K}^2)$ vanishes if and only if the associated equivariant harmonic map $f: \widetilde X \to \frak{X}$ is minimal (recall that $\frak{X}=G/K$ is the symmetric space of $G$).

Maximal representations are Anosov \cite{BILW}, so in particular well-displacing \cite{wienhardanosov}. By a theorem of Labourie \cite{labourieenergy}, there exists a Riemann surface structure $X\in\mathcal{T}(\S)$ such that the associated equivariant harmonic map $f:\widetilde X \to \frak{X}$ is a minimal immersion. In particular, the corresponding $G$ Higgs bundle looks like
\[(\mathcal{E},\Phi)= \xymatrix{
\mathcal{I}\mathcal{K} \ar@/_/[r]_1 & \mathcal{I}  \ar@/_/[r]_1  & \mathcal{I}\mathcal{K}^{-1}  \ar@/^/[ld]^{\beta} \\
& \mathcal{V} \ar@/^/[ul]^{\beta^\dagger} }.\]
Such a $G$ Higgs bundle corresponds to very special point in the moduli space $\mathcal{M}(X,G)$. Indeed, it is fixed by the action of the fourth root of unity by multiplication of the Higgs field $\Phi$. As a consequence (see \cite{KatzMiddleInvCyclicHiggs,collierthesis}), the splitting $\mathcal{E}=\mathcal{I}\mathcal{K} \oplus \mathcal{I}\mathcal{K}^{-1}\oplus \mathcal{I}\oplus \mathcal{V}$ is orthogonal with respect to the Hermitian metric $h$ solving the Hitchin equations.

In particular, the anti-linear involution $\lambda : \mathcal{E}\to\mathcal{E}$ such that $h(x,y)=q(x,\lambda y)$ preserves $\mathcal{I}\mathcal{K} \oplus \mathcal{I}\mathcal{K}^{-1},~\mathcal{I}$ and $\mathcal{V}$. The bundle $E_\rho=\text{Fix}(\lambda)$ thus splits orthogonally as
$$E_\rho=U\oplus \ell \oplus V,$$
where $U$ is a positive definite rank $2$ sub-bundle, $\ell$ is a negative-definite line sub-bundle and $V=(U\oplus\ell)^\bot$.

The line sub-bundle $\ell\subset E_\rho$ corresponds to a $\rho$-equivariant map $u: \widetilde X \to \H^{2,n}$. Using the explicit decomposition of the $G$ Higgs bundle $(\mathcal{E},\Phi)$, one can show that $u$ is harmonic and conformal, hence a maximal immersion (see \cite[Section 3.3]{CTT} for more details).
\end{proof}

\begin{rem}\label{r:fromminimaltomaximal}
The previous proof gives a way to associate a $\rho$-equivariant maximal surface $u: \widetilde X\to \H^{2,n}$ to a $\rho$-equivariant minimal surface $f: \widetilde X \to \frak{X}$. Moreover, the induced metric $u^*g_{\H^{2,n}}$ is conformal to $f^*g_{\frak{X}}$. In Section \ref{s:gaussmap}, we use the Gauss map to reconstruct the minimal surface $f: \widetilde X \to \frak{X}$ from the maximal one.
\end{rem}

\subsection{Uniqueness}

Here, we sketch the proof of the uniqueness part in Main Theorem:

\begin{prop}
If $\rho:\Gamma \to G$ is a maximal representation, then there exists at most one $\rho$-equivariant maximal space-like surface $u:\widetilde\S\longrightarrow \H^{2,n}$.
\end{prop}

\begin{proof}[Sketch of proof]
Suppose there exist $2$ different maximal surfaces $S_1,S_2 \hookrightarrow \H^{2,n}$ on which $\rho(\Gamma)$ acts co-compactly.

The idea is to define a causal distance between $S_1$ and $S_2$. This distance is achieved by a time-like geodesic segment. In time-like directions, negatively curved pseudo-Riemannian manifolds behave like their positively curved Riemannian counterpart. In particular, one can use a maximum principle to show a contradiction.

More precisely, one starts to show that the lift of a $\rho(\Gamma)$-invariant space-like surface to the double cover $\widehat\H^{2,n}$ has two connected components. By choosing a lift of $S_i$ to $\widehat\H^{2,n}$ (for $i=1,2$), one can define a function
\[\begin{array}{llll}
B: & S_1 \times S_2 & \longrightarrow & \R \\
 & (u,v) & \longmapsto & \langle u,v \rangle
\end{array}.\]
Here, $\langle .,.\rangle$ is the restriction of the scalar product in $\R^{2,n+1}$. The function $\arccos (-B)$ is the causal distance function.

By choosing the lifts of $S_i$ in a coherent way, one can show that $B(u,v)<0$ for any $(u,v)\in S_1\times S_2$. By co-compactness of the action of $\rho(\Gamma)$, there exists a couple $(u_0,v_0)\in S_1\times S_2$ such that $B(u_0,v_0)$ is maximum, and moreover $B(u_0,v_0)>-1$ (that is, the geodesic segment passing through $u_0$ and $v_0$ is time-like).

Given unit vectors $\dot{u}_0\in T_{u_0}S_1$ and $\dot{v}_0\in T_{v_0}S_2$, one can consider geodesic paths $\big( u(t)\big)_{t\in(-\epsilon,\epsilon)}$ and $\big( v(t)\big)_{t\in(-\epsilon,\epsilon)}$ in $S_1$ and $S_2$ such that $u'(0)=\dot{u}_0$ and $v'(0)=\dot{v}_0$. The second derivative $\ddot{B}(u_0,v_0)$ of $B\big(u(t),v(t)\big)$ at $t=0$ is thus given by:
\[\ddot{B}(u_0,v_0) = 2\langle \dot{u}_0,\dot{v}_0\rangle + B(D_{\dot{u}_0}\dot{u}_0,v_0)+B(u_0,D_{\dot{v}_0}\dot{v}_0).\]
Here, $D$ is the usual connection on $\R^{2,n+1}$. Because $\widehat\H^{2,n}\subset \R^{2,n+1}$ is umbilical, the connection $D$ evaluated at a point $p\in \widehat\H^{2,n}$ decomposes as:
\[(D_XY)_p = (\nabla^{\widehat\H^{2,n}}_XY)_p + \langle x,y \rangle p,\]
where $\nabla^{\widehat\H^{2,n}}$ is the Levi-Civita connection of $\H^{2,n}$ and $X,Y$ are vector fields on $\widehat\H^{2,n}$. In the same way, the connection $\nabla^{\widehat\H^{2,n}}$ along $S_i$ decomposes as:
\[ \nabla^{\widehat\H^{2,n}}_XY= \nabla^i_X Y + \textrm{II}_i(X,Y),\]
where $\nabla^i$ and $\textrm{II}_i$ are respectively the Levi-Civita connection and the second fundamental form on $S_i$.

In particular, we get
\[\ddot{B}(u_0,v_0) = 2\langle \dot{u}_0,\dot{v}_0\rangle + \langle\textrm{II}_1(\dot{u}_0,\dot{u}_0),v_0\rangle + \langle u_0,\textrm{II}_2(\dot{v}_0,\dot{v}_0)\rangle + 2 B(u_0,v_0). \]

Because $S_1$ and $S_2$ are maximal, the quadratic forms $\beta_1: w \mapsto \langle\textrm{II}_1(w,w),v_0\rangle$ and $\beta_2: w \mapsto \langle u_0, \textrm{II}_2(w,w)\rangle$ are trace-less. Writing their opposite eigenvalues by $\pm \lambda_1$ and $\pm \lambda_2$ respectively, we can assume $\lambda_1\geq\lambda_2\geq 0$. By choosing $\dot{u}_0$ such that $\beta_1(\dot{u}_0)=\lambda_1$, we get that
\[\langle\textrm{II}_1(\dot{u}_0,\dot{u}_0),v_0\rangle + \langle u_0,\textrm{II}_2(\dot{v}_0,\dot{v}_0)\rangle\geq 0.\]
Now, using the fact that $\text{span}(u_0,v_0)$ has signature $(0,2)$ and that the $S_i$ are space-like, one can show that $\langle \dot{u}_0,\dot{v}_0\rangle \geq 1$, and we get
\[\ddot{B}(u_0,v_0)\geq 2 + 2B(u_0,v_0)> 0. \]
The maximum principle then contradicts the maximality of $B(u_0,v_0)$. 
\end{proof}

\subsection{Gauss map}\label{s:gaussmap}

We explain here how to reconstruct a minimal surface $f: \widetilde X\to \frak{X}$ from a maximal one $u: \widetilde\X \to \H^{2,n}$. 

Given a space-like surface $S\hookrightarrow \H^{2,n}$, the exponential map sends the fiber of the normal bundle $NS$ at a point $p\in S$ onto a totally geodesic time-like projective space $\mathbb{RP}^{0,n}\subset \H^{2,n}$. Such a time-like projective space in the projectivization of a $\R^{0,n+1}$ in $\R^{2,n+1}$. In particular, the set on totally geodesic time-like projective spaces in $\H^{2,n}$ is canonically identified with the symmetric space $\frak{X}=G/K$. It defines the \textit{Gauss map}:
\[\mathcal{G}: S \longrightarrow \frak{X}.\]

By a theorem of Ishihara \cite{ishihara}, the Gauss map of a maximal surface is harmonic. In particular, the Gauss map of the $\rho$-equivariant maximal space-like surface $u: \widetilde X \to \H^{2,n}$ is the $\rho$-equivariant harmonic map $f: \widetilde X \to \frak{X}$ given by Corlette's theorem. By construction, this harmonic map is conformal and so is a minimal surface.

\section{Consequences}

\subsection{Characterization of the maximal components}

For $G=\SO_0(2,n+1)$ with $n\geq 2$, the space $\Rep^{max}(\Gamma,G)$ of maximal representations is not connected. It has
\begin{itemize}
\item $2.2^{2g}$ connected components for $n\geq 3$ (\cite{garcia-prada}).
\item $2.(2^{2g}-1)+4g-3$ connected components for $n=2$ (\cite{gothen}).
\end{itemize}

Our Main Theorem gives a nice interpretation of all these components: given $\rho\in\Rep^{max}(\Gamma,G)$, let $u: \widetilde\S \to \H^{2,n}$ be the unique $\rho$-equivariant space-like maximal surface. The pull-back bundle $u^*T\H^{2,n}$ has the following orthogonal splitting:
\[u^*T\H^{2,n} = T\widetilde\S \oplus N \widetilde \S.\]
Because $u$ is space-like, the normal bundle $N\S:= N\widetilde\S /\rho(\Gamma)$ is a (negative definite) $\text{O}(n)$-bundle over $\S$. By uniqueness of $u$, the topological invariants of $N\S$ give invariants of the representation $\rho$.

Recall the following:
\begin{itemize}
\item For $n\geq 3$, the topology of an $\text{O}(n)$-bundle over $\S$ is classified by its first and second Stiefel-Whitney classes $w_1\in H^1(\S,\mathbb{Z}/2\mathbb{Z})$ and $w_2\in H^2(\S,\mathbb{Z}/2\mathbb{Z})$ (giving $2.2^{2g}$ possibilities).
\item For $n=2$, an $\text{O}(2)$-bundle with zero first Stiefel-Whitney class reduces to an $\SO(2)$ bundle and the topology of an $\SO(2)$-bundle over $\S$ is classified by its degree $d\in \mathbb{Z}$. In our case, one can show that the degree of $N\S$ is between $0$ and $4g-4$ (giving a total of $2(2^{2g}-1)+4g-3$ possibilities). 
\end{itemize}

It turns out that each one of these topology arises as the topology of the normal bundle $N\S$, and these new invariants characterize the connected components of $\Rep^{max}(\Gamma,G)$.
 
\subsection{Labourie's conjecture}

Labourie's conjecture \cite{labourieenergy} was originally stated for Hitchin representations:

\begin{conj}[Labourie]
Given a Hitchin representation $\rho: \Gamma \to G$ into a real split Lie group $G$, there exists a unique $\rho$-equivariant minimal immersion $f: \widetilde\S \to \X$, where $\X=G/K$ is the symmetric space of $G$.
\end{conj}

The main motivation was to give a parametrization of the Hitchin component as a complex vector bundle over the Teichm\"uller space. This conjecture was proved by Schoen \cite{schoen} for $\SO_0(2,2)$, by Loftin \cite{loftin} and independently Labourie \cite{labouriecubic} for $\textrm{SL}(3,\R)$ and by Labourie \cite{labouriecyclic} for any Hitchin representation in a rank 2 real split Lie groups.

It is natural to extend the conjecture to maximal representations. Our Main Theorem implies the following:

\begin{prop}\label{p:Labourieconjecture}
Given a maximal representation $\rho: \Gamma \to G$, where $G$ is a Hermitian Lie group of rank $2$, there exists a unique $\rho$-equivariant minimal surface $f: \widetilde\S \to \frak{X}$, where $\frak{X}=G/K$ is the Riemannian symmetric space of $G$.
\end{prop}

\begin{proof}
By Corollary \ref{cor:maximalrepinrank2}, it suffices to restrict ourselves to the case $G=\SO_0(2,n+1)$.

As explained in Section \ref{s:existence}, the existence was known. For the uniqueness, note that our construction associate a $\rho$-equivariant space-like maximal surface $u: \widetilde\S \to \H^{2,n}$ to any $\rho$-equivariant minimal surface $f: \widetilde\S \to \frak{X}$. Recall also (see Remark \ref{r:fromminimaltomaximal}) that the induced metric $u^*g_{\H^{2,n}}$ is conformal to $f^*g_{\frak{X}}$. By the uniqueness part of the Main Theorem, if there exist two $\rho$-equivariant minimal surface $f_1,f_2: \widetilde\S \to \frak{X}$, the induced metrics $f^*_ig_{\frak{X}}$ are conformal. However, minimal immersion are harmonic, so the uniqueness part of Corlette's theorem implies that $f_1=f_2$.
\end{proof}

\subsection{Complex structure on $\Rep^{max}(\Gamma,G)$}

A main motivation to Labourie's conjecture for Hitchin representations is to get a Mapping Class Group invariant parametrization of the Hitchin component by a holomorphic vector bundle over the Teichm\"uller space. In particular, if the conjecture is true, one would get a natural complex structure on the Hitchin component.

For the case of maximal representations, using Higgs bundles theory, one could parametrize the components of maximal representations as a bundle over the Teichm\"uller space. In general, the fiber of this bundle would be singular complex manifolds.

In a recent work, D. Alessandrini and B. Collier \cite{PSp4maximalRepsAC} adapted a construction of Simpson \cite{simpsonI,simpsonII} to obtain the following:

\begin{prop}[Alessandrini-Collier]
The space of maximal representations $\Rep^{max}(\Gamma,G)$ of $\Gamma$ into a Hermitian Lie group of rank $2$ admits a Mapping Class Group invariant parametrization by a bundle $\mathcal{B} \to \T(\S)$. Moreover, $\mathcal{B}$ carries a natural complex structure.
\end{prop}

\begin{proof}[Sketch of proof]
Consider the Universal Moduli space of maximal $G$ Higgs bundle
\[\pi: \mathcal{M}^{max}(\mathcal{U},G) \longrightarrow \T(\S),\]
whose fiber $\pi^{-1}(X)$ is the moduli space of $G$ Higgs bundles over $X$ parametrizing $\Rep^{max}(\Gamma,G)$ via the non-Abelian Hodge correspondence.

Recall that a $G$ Higgs bundle $(\mathcal{E},\Phi)$ is associated to an equivariant minimal surface $f: \widetilde X \to \frak{X}$ if and only if $\text{tr}(\Phi^2)=0$ (see Corollary \ref{cor:maximalhiggsbundles}).

Proposition \ref{p:Labourieconjecture} implies that the non-Abelian Hodge correspondence gives a bijection between $\Rep^{max}(\Gamma,G)$ and 
$$\mathcal{B}:=\big\{(\mathcal{E},\Phi,X)\in \mathcal{M}^{max}(\mathcal{U},G),~\text{tr}(\Phi^2)=0 \big\}\subset \mathcal{M}^{max}(\mathcal{U},G).$$
 It thus gives an equivariant parametrization of $\Rep^{max}(\Gamma,G)$.

Using a result of Simpson \cite{simpsonI,simpsonII}, Alessandrini and Collier showed that $\mathcal{M}^{max}(\mathcal{U},G)$ carries a canonical complex structure. Moreover, the sub-bundle $\mathcal{B}\subset \mathcal{M}^{max}(\mathcal{U},G)$ is given by the equation $\text{tr}(\Phi^2)=0$ which is a holomorphic equation. It follows that $\mathcal{B}$ is a holomorphic sub-manifold. 
\end{proof}

Finally, let us just remark that the space $\Rep^{max}(\Gamma,G)$ of maximal representations carries a natural symplectic form called the \textit{Goldman symplectic form}. It would be interesting to understand if this symplectic form is compatible with the complex structure, in order to try to construct a K\"ahler structure on $\Rep^{max}(\Gamma,G)$.,
{\small
\bibliographystyle{alpha}
\bibliography{maximalreps.bbl}}

\end{document}